\documentclass[12pt]{article}
\usepackage{amsmath,amsthm}
\usepackage{amssymb}
\usepackage{epsfig}
\usepackage{tikz}
\usepackage{graphicx}
\usepackage{url}
\usepackage{authblk}
\theoremstyle{plain}
\newtheorem{theorem}{Theorem}[section]
\newtheorem{corollary}[theorem]{Corollary}

\theoremstyle{definition}

\newtheorem{remark}[theorem]{Remark}
\newtheorem{example}[theorem]{Example}
\newtheorem{observation}[theorem]{Observation}
\newtheorem{problem}[theorem]{Problem}

\newcommand{\PSL}{\mathop{\mathrm{PSL}}}
\newcommand{\PSU}{\mathop{\mathrm{PSU}}}
\newcommand{\Sz}{\mathop{\mathrm{Sz}}}
\newcommand{\PSp}{\mathop{\mathrm{PSp}}}
\newcommand{\SL}{\mathop{\mathrm{SL}}}
\newcommand{\SU}{\mathop{\mathrm{SU}}}

\begin{document}
  \setcounter{Maxaffil}{3}
   \title{On finite groups whose power graph is a cograph}
   \author[a]{Peter J. Cameron\thanks{pjc20@st-andrews.ac.uk}}
   \author[b]{Pallabi Manna\thanks{mannapallabimath001@gmail.com}}
   \author[b]{Ranjit Mehatari\thanks{ranjitmehatari@gmail.com, mehatarir@nitrkl.ac.in}}
   \affil[a]{School of Mathematics and Statistics,}
   \affil[ ]{University of St Andrews,}
   \affil[ ]{North Haugh, St Andrews, Fife, KY16 9SS, UK\\}
   \affil[ ]{ }
   \affil[b]{Department of Mathematics,}
   \affil[ ]{National Institute of Technology Rourkela,}
   \affil[ ]{Rourkela - 769008, India}
   \maketitle
   \begin{abstract}
    A $P_4$-free graph is called a cograph. In this paper we partially characterize finite groups whose power graph is a cograph. As we will see, this problem
is a generalization of the determination of groups in which every element has
prime power order, first raised by Graham Higman in 1957 and fully solved very recently.

First we determine all groups $G$ and $H$ for which the power graph of $G\times H$ is a cograph.  We show that groups
 whose power graph is a cograph can be characterised by a condition only
involving elements whose orders are prime or the product of two (possibly
equal) primes. Some important graph classes are also taken under consideration. For finite simple groups we show that in most of the cases their power graphs are not cographs: the only ones for which the power graphs are cographs are
certain groups $\PSL(2,q)$ and $\Sz(q)$ and the group $\PSL(3,4)$. However, a
complete determination of these groups involves some hard number-theoretic problems.
   \end{abstract}

   \section{Introduction}
   There are various graphs we can define for a group using different group properties \cite{Cameron1}.
   These graphs include the commuting graph, the generating graph, the power graph, the enhanced power graph, deep commuting graph, etc. The power graphs were first seen in early 2000's as the undirected power graphs of semigroups \cite{Kelarev}.
For a semigroup $S$, the \emph{directed power graph} of $S$, denoted by
$ \vec{P}(S)$,  is a directed graph with vertex set $V(\vec{P}(S))=S$; and two distinct vertices $x$ and $y$ are having an arc
$x\to y$ if $y$ is a power of $x$.

The corresponding undirected graph is called  the
\emph{undirected power graph} of $S$, denoted by $P(S)$. The undirected
power graph of a semigroup was introduced by
Chakrabarty et~al.~\cite{Chakrabarty} in 2009. So the undirected power graph of
$S$ is the graph with vertex set $V(P(S))=S$,
with an edge between two vertices $u$ and $v$ if $ u\ne v$ and either $v$ is
a power of $u$ or $u$ is a power of~$v$. These concepts are defined for groups
as a special case of semigroups. In the sequel, we only consider groups; ``power graph'' will
mean ``undirected power graph'', and all the groups in this paper are finite.
   
   The power graphs were well studied in literature \cite{aacns,Abawajy, Cameron, cg,cgj,cj,Chakrabarty, Chattopadhyay,Doostabadi, Ma,cmm}. We find several research papers in which researchers give complete or partial characterization of different graph parameters for the power graphs. We mention few notable works in this context:
   \begin{itemize}
       \item $P(G)$ is a complete graph if and only if either $G$ is trivial or a cyclic group of prime power order. (Chakrabarty \emph{et~al.} \cite{Chakrabarty})
\item
$P(G)$ is always connected and we can compute the number of edges in $P(G)$ by the formula
$|E(P(G))|=\dfrac{1}{2}\Big[\sum_{a\in G}(2o(a)-\phi(o(a))-1)\Big]$.
\item
The power graph
of a finite group is Eulerian if and only if  $G$ has odd number of elements.
\item
Curtin \emph{et~al.} \cite{Curtin} introduced the concept of proper power graphs. They determine
the diameter of the proper power graph of  $S_n$.

\item
Chattopadhyay \emph{et~al.} \cite{Chattopadhyay} have provided suitable bounds for the
vertex connectivity $P(G)$ where $G$ is a cyclic group.
\item
Cameron \cite{Cameron} proved that, for any two finite groups $G_1$ and $G_2$, if power graphs of $G_1$ and $G_2$ are isomorphic then  $\vec{P}(G_1)$ and $\vec{P}(G_2)$ are also isomorphic.

   \end{itemize}
   In our previous paper \cite{cmm}, we partially characterized finite groups whose power graphs forbid certain induced subgraphs. These subgraphs include
$P_4$ (the path on $4$ vertices); $C_4$ (the cycle on 4 vertices); $2K_2$ (the complement of $C_4$); etc. A graph forbidding $P_4$ is called a cograph.
In other words,  a graph $\Gamma$ is a \emph{cograph} if it does not contain
the $4$-vertex path as an induced subgraph.
Cographs have various important properties. For
example, they form the smallest class of graphs containing the $1$-vertex
graph and closed under complementation and disjoint union. (The complement of
a connected cograph is disconnected.)
 See~\cite{Br,Cameron1} for more about these concepts.

We will use the term \emph{power-cograph group}, sometimes abbreviated to
\emph{PCG-group}, for a finite group whose power graph is a cograph.
 
In \cite{cmm}, we completely characterized finite nilpotent power-cograph
groups. We proved the following theorem:

   \begin{theorem}[\cite{cmm}, Theorem 3.2]\label{t:nilp_cograph}
Let $G$ be a finite nilpotent group. Then $P(G)$ is a cograph if and only if
either $|G|$ is a prime power, or $G$ is cyclic of order $pq$ for distinct
primes $p$ and $q$.
\end{theorem}

For a given group $G$, the power graph of any subgroup of $G$ is an induced subgraph of $P(G)$. Thus if power graph of a group is a cograph then the power graph of any of its subgroups is also a cograph. In other words, the class of finite
power-cograph groups is subgroup-closed.

For that reason, we have a necessary condition for a group to be a power-cograph
group: any nilpotent subgroup of such a group is either a $p$-group or isomorphic to $C_{pq}$, where $p$ and $q$ are distinct primes.  So if $G$ has a nilpotent subgroup which is neither a $p$-group nor isomorphic to $C_{pq}$, then $P(G)$ is not a cograph. In our previous paper, we have asked the following question:
Classify the finite groups $G$ for which $P(G)$ is a cograph.
In this paper we provide further results towards the answer to this question. 

We now give several equivalent conditions on a finite group which are known to
imply that the power graph is a cograph. First we require a few definitions.
\begin{itemize}
\item
For a finite group $G$, Let $\pi(G)$ denote the set of all prime divisors of $|G|$. The \emph{prime graph} or \emph{Gruenberg--Kegel graph} of $G$ is a graph with $V=\pi(G)$ and two distinct elements $p$ and $q$ of $\pi(G)$ are connected if and only if $G$ contains an element of order $pq$.
\item
The \emph{enhanced power graph} of $G$ is the graph with vertex set $G$, in
which vertices $x$ and $y$ are joined if there exists $z\in G$ such that both
$x$ and $y$ are powers of $z$. Clearly the power graph is a spanning subgraph
of the enhanced power graph.
\item
The group $G$ is an \emph{EPPO group} if every element of $G$ has prime power
order.
\end{itemize}

\begin{theorem}
\label{prime:cograph} 
For a finite group $G$, the following conditions are equivalent:
\begin{enumerate}
\item $G$ is an EPPO group;
\item the Gruenberg--Kegel graph of $G$ has no edges;
\item the power graph of $G$ is equal to the enhanced power graph.
\end{enumerate}
If these conditions hold, then the power graph of $G$ is a cograph.
\end{theorem}

For the equivalence of (a)--(c), see Aalipour \emph{et al.}~\cite{aacns}. If these hold, then any edge of the power graph not containing the identity
joins elements whose orders are powers of the same prime; so the reduced
power graph of $G$ (obtained by removing the identity) is a disjoint union
of reduced power graphs of groups of prime power order, which are cographs by
Theorem \ref{t:nilp_cograph}.
The class of EPPO groups was first investigated (though not under that name)
by Graham Higman in 1957~\cite{Higman}, but the complete determination of these
groups only appeared in a paper not yet published~\cite{cm}.

We note that the condition that $G$ is a power-cograph group does not imply (a)--(c).
Moreover two groups may have the same prime graph, yet one and not the other is
a power-cograph group. For example, consider $G_1=C_{12}$ and $G_2=D_6$.

In this paper we explore various graph classes and try to identify whether their power graphs are cographs or not. First we discuss direct product two groups. We are able to identify certain solvable groups  whose power graph is a cograph. Finally we consider finite simple groups. Our result is as follows:

\begin{theorem}
Let $G$ be a non-abelian finite simple group. Then $G$ is a power-cograph group
if and only if one of the following holds:
\begin{enumerate}
\item $G=\PSL(2,q)$, where $q$ is an odd prime power with $q\ge5$, and each of
$(q-1)/2$ and $(q+1)/2$ is either a prime power or the product of two
distinct primes;
\item $G=\PSL(2,q)$, where $q$ is a power of $2$ with $q\ge4$, and each of
$q-1$ and $q+1$ is either a prime power or the product of two distinct primes;
\item $G=\Sz(q)$, where $q=2^{2e+1}$ for $e\ge2$, and each of $q-1$,
$q+\sqrt{2q}+1$ and $q-\sqrt{2q}+1$ is either a prime power or the product of
two distinct primes;
\item $G=\PSL(3,4)$.
\end{enumerate}
\end{theorem}

We end the introduction with some remarks about this theorem. In the first
three cases, determining precisely which groups occur is a purely 
number-theoretic problem which is likely to be quite difficult. For example,
the values of $d$ (at least~$2$) for which the power graph of $\PSL(2,2^d)$
is a cograph for $d\le200$ are $1$, $2$, $3$, $4$, $5$, $7$, $11$, $13$, $17$,
$19$, $23$, $31$, $61$, $101$, $127$, $167$, and $199$, and the values of $e$
(at least~$1$) for which the power graph of  $\Sz(2^{2e+1})$ is a cograph for
$e\le100$ are $1$, $2$, $3$, $4$, $5$, $6$, $8$, $44$.

\begin{problem} 
Are there infinitely many non-abelian finite simple groups
$G$ which are power-cograph groups?
\end{problem}

\medskip

Secondly, there is a big gap between finding the simple groups satisfying the
condition and finding all groups. This can be seen in the somewhat similar
property of being an EPPO group, where the list of simple EPPO groups follows
from the work of Suzuki~\cite{suzuki1,suzuki2} but the complete determination
of these groups is much more recent.

   \section{Direct products}

   Recall that a finite nilpotent group can be written as a direct product of its Sylow  subgroups. Thus if $P(G\times H)$ is a cograph  then, by using Theorem \ref{t:nilp_cograph}, we have only a few choices for the orders of $G$ and $H$. The following theorem gives a complete characterization for all direct products $G\times H$ such that $P(G\times H)$ is a cograph.

   \begin{theorem}
   \label{product:cographs}
Let $G$ and $H$ be non-trivial groups. Then $G \times H$ is a power-cograph
group if and only if one of the following holds:
\begin{enumerate}
\item the orders of $G$ and $H$ are powers of the same prime;
\item $G$ and $H$ are cyclic groups of distinct prime orders;
\item there are primes $p$ and $q$ and an integer $m \geq 1$ such that
$ q^{m}\mid (p-1)$; one of $G$ and $H$ is a cyclic group of order $q$, and the other is the non-abelian group
\[\langle  a,b: a^{p}=1, b^{q^m}=1, b^{-1}ab=a^k\rangle,\]
where $k$ is an integer with multiplicative order $q^{m}$ $(\text{mod }p)$.
\end{enumerate}
\end{theorem}

\begin{proof}
Let $P(G\times H)$ be a cograph. If the orders of $G$ and $H$ are each
divisible by exactly one prime then $G\times H$ is nilpotent; by Theorem 
\ref{t:nilp_cograph}, we have one of the first two cases. So we can suppose
that at least one of $G$ and $H$ has order divisible by two primes. 

Suppose that $|G|$ is divisible by primes $p$ and $q$. Then no prime except
possibly $p$ or $q$ can divide $H$. For suppose that $r\mid|H|$; let $a$ and
$b$ be elements of orders $p$ and $q$ in $G$, and $c$ an element of order $r$
in $H$. Then $(b,bc,c,ac)$ is an induced $P_4$ in $P(G\times H)$.

It follows that at most two primes divide each of $|G|$ and $|H|$.

Suppose first that both $p$ and $q$ divide each of $|G|$ and $H$. The direct
product of a Sylow $p$-subgroup of $G$ and a Sylow $q$-subgroup of $H$ is
nilpotent. So Theorem~\ref{t:nilp_cograph} implies that $|G|=|H|=pq$. Each is
non-abelian; so, without loss of generality, $q\mid p-1$. Let $a$ be an
element of order $p$ in $G$, and $b$, $c$ elements of order $q$ in $H$ which
are not joined in the power graph. Then $(b,ab,a,ac)$ is an induced $P_4$.

In the remaining case, one of $G$ and $H$ (say $G$, without loss of generality)
is a group of prime power order $q^m$. By assumption, $H$ is not a $q$-group,
and so contains a subgroup $P$ of prime order $p \neq q$. Then $G \times P$ is
a nilpotent subgroup of $G \times H$ with forbidden structure, unless $|G| = q$.
Now $p$ must divide $|H|$ to the first power only, and $|H| = pq^m$ for some
$m \geq 1$.

We claim next that $H$ has a normal Sylow $p$-subgroup $P$. For suppose
not, and let $b$ and $c$ be elements of order $p$ not adjacent in the power 
graph, and $a$ a non-identity element of $G$. Then $(b, ab, a, ac)$ is an
induced path in the power graph of $G \times H$, a contradiction.
As before, we conclude that $C_H(P) = P$, and so the Sylow $q$-subgroup of
$H$ (which is a complement to $P$) is cyclic of order dividing $p-1$. This 
yields the claimed structure for $G$ and $H$.

\medskip

For the converse, we begin with some preliminary remarks. If $x,y$ are elements
of a group $G$, and $x\to y$ in the directed power graph of $G$, then $y$ is
a power of $x$, and so $o(y)\mid o(x)$. If also $y\not\to x$, then the
divisibility is proper. Moreover, since $\to$ is a transitive relation, if
$(a,b,c,d)$ is an induced power graph in $P(G)$, then either
$a\to b\gets c\to d$ or the reverse.

So let
\[H=\langle a,b:a^p=b^{q^m}=1, b^{-1}ab=a^k\rangle,\]
and $G=C_q=\langle c\rangle$. Then all non-identity elements in $G\times H$
have orders a power of $q$, $p$, or $pq$. Also, if $(g,h)$ has order a power
of $q$, then $(g,h)^q=(1,h^q)$.

Suppose if possible that $(x_1,x_2,x_3,x_4)$ is an induced path in
$P(G\times H)$, where $x_i=(g_i,h_i)$, and suppose that
$x_1\to x_2\gets x_3\to x_4$ in $\vec{P}(G)$. Then none of $x_1,\ldots,x_4$ is
the identity, so $x_1$ and $x_3$ each have order a power of $q$ or $pq$.

If $x_3$ has order a power of $q$, then $g_2=g_4=1$ and $h_2$ and $h_4$ belong
to the same (cyclic) Sylow $q$-subgroup of $H$, so $x_2$ and $x_4$ are adjacent
in $P(G\times H)$, a contradiction. So $x_3$ has order $pq$, and $x_2$ and
$x_4$ have orders $p$ or $q$. If $x_2$ has order $p$, then $x_1$ has order
$pq$. On the other hand, if $x_2$ has order $q$, then $g_2\ne1$, so $x_2$
cannot be a $q$th power, so again $x_1$ has order $pq$. But this implies
that $x_1$ and $x_3$ are joined.
\end{proof}

\begin{remark}
The converse of the above is not true in general.
Consider $G=C_{4}$ and $H=C_{6}$. Then by Theorem \ref{t:nilp_cograph}, both $P(G)$ and $P(H)$ are cographs, whereas $P(G \times H)$ is not a cograph.
\end{remark}

\section{Minimal non-power-cograph groups}

Let $\mathcal{C}$ be the class of finite groups $G$ for which $P(G)$ is
a cograph. As noted earlier, $\mathcal{C}$ is subgroup-closed; so it can be
characterised by finding all minimal non-$\mathcal{C}$ groups.

\begin{theorem}
\label{minimal:cograph}
Let $G$ be a finite group. Then $P(G)$ is not a cograph if and only if
$G$ contains elements $g$ and $h$ with orders $pr$ and $pq$ respectively,
where $p,q,r$ are prime numbers and $p\ne q$, such that
\begin{enumerate}
\item $g^r=h^q$;
\item if $q=r$, then $g^p\notin\langle h^p\rangle$.
\end{enumerate}
\end{theorem}

\begin{proof}
Let $G$ be a minimal non-$\mathcal{C}$ group. 
Suppose first that $G$ is abelian. By Theorem~\ref{t:nilp_cograph}, it has order the
product of three primes which are not all equal. We distinguish three cases.
\begin{itemize}\itemsep0pt
\item Suppose that $|G|=pqr$ where $p,q,r$ are all distinct.
Then $G$ is cyclic; say $G=\langle x\rangle$. Now if we put $g=x^q$
and $h=x^r$, we see that the conditions of the theorem are satisfied.
\item Suppose that $|G|=p^2q$, and that the Sylow $p$-subgroup of $G$ is
cyclic, generated by $g$. Let $z$ be an element of order $q$, and $h=g^pz$.
Take $r=p$ in the conditions of the theorem.
\item Finally, suppose that $|G|=p^2q$ and the Sylow $p$-subgroup is elementary
abelian, generated by $x$ and $y$. Let $z$ be an element of order $q$.
Now take $g=xz$ and $h=yz$. Then $g$ and $h$ have order $pq$; $g^p=z^p=h^p$,
but $g^q=x^q\notin\langle x^p\rangle$. So these elements satisfy the conditions of
the theorem, if we take $r=q$ and reverse the roles of $p$ and $q$.
\end{itemize}
So we can suppose that $G$ is nonabelian.

Since $P(G)$ is not a cograph, there is an induced path $(a,b,c,d)$ in 
$P(G)$. As we saw in the proof of Theorem~\ref{product:cographs}, we may assume
that $a\to b\gets c\to d$ in $\vec{P}(G)$.

Now $\langle c\rangle$ is a cyclic group and contains $b$ and $c$. Since
$G$ is nonabelian, it is a proper subgroup, and hence its power graph is a
cograph. So the order of $c$ is either a prime power or of the form $pq$
where $p$ and $q$ are distinct primes. The former case is impossible. For
the power graph of a cyclic group of prime power order is complete, but
$b$ is not joined to $d$. So the order of $c$ is $pq$, with $p\ne q$. We may
suppose without loss that $b=c^q$ has order $p$ while $d=c^p$ has order $q$.

Now consider the element $a$. We know that the order of $a$ is divisible
by $p$ (the order of $b$). By replacing $a$ by a power of itself, we can
assume that the order of $a$ is $pr$, where $r$ is a prime which may or may not
be equal to $p$. (This power is still joined to $b$, but it cannot be joined to
$d$. For if $a$ and $d$ are joined, then
$d\in\langle a\rangle\cap\langle c\rangle=\langle b\rangle$, contradicting
the fact that $d$ has order $q$ whereas $b$ has order $p$. Also $a$ cannot be
joined to $c$, for this would imply that $a\to c$ and hence $a\to d$.)

We have now verified all the conditions of the theorem.

\medskip

Conversely, if these conditions hold, then $(g,g^r=h^q,h,h^p)$ is an induced
path of length $3$, so $P(G)$ is not a cograph. 
\end{proof} 

\begin{remark}
A minimal non-PCG group has nontrivial centre. For such a group is generated
by elements $g$ and $h$ as in the theorem, and $g^r=h^q$ is in the centre.
\end{remark}

\begin{corollary}
Let $G$ be a finite group. Let $P_2(G)$ be the set of non-identity elements
of $G$ whose orders are either prime or the product of two (not necessarily
distinct) prime numbers. Then $P(G)$ is a cograph if and only if the
induced subgraph on $P_2(G)$ is a cograph.
\end{corollary}

Here is an application, which we will require later. Suppose that $G$ is a
finite group containing elements $a$ of order~$4$ and $b$ of order~$6$ such
that $a^2$ and $b^3$ are conjugate. Replacing $b$ by a conjugate, we may
assume that $a^2=b^3$. Now the theorem above implies that $G$ is not a
power-cograph group. These conditions can be verified for the simple groups
$M_{11}$ and $\PSU(3,8)$ using the $\mathbb{ATLAS}$ of finite groups
\cite{Conway}. We will use this argument several times, so we refer to it as
the \emph{$4$-$6$ test}.

\section{Examples}

Below we let $P^*(G)$ be the \emph{reduced power graph} of $G$, the induced
subgraph on the set $G^\#=G\setminus\{1\}$. Note that $P(G)$ is a cograph
if and only if $P^*(G)$ is a cograph.

We also make the following observation.

\begin{theorem}\label{t:maxcyclic}
Let $G$ be a finite group in which any two distinct maximal cyclic subgroups
intersect in the identity. Then $P(G)$ is a cograph if and only if the orders
of the maximal cyclic subgroups are either prime powers or products of two
distinct primes.
\end{theorem}

\begin{proof}
Every edge of $P(G)$ is contained in a maximal cyclic subgroup of $G$. The
hypothesis implies that $P^*(G)$ is the union of $P^*(C)$ as $C$ runs over
the maximal cyclic subgroups of $G$.
\end{proof}

\begin{theorem}
\label{symm:cograph}
The symmetric group $S_n$ on $n$ symbols is a power-cograph group if and only
if $ n \leq 5$.
\end{theorem}

 \begin{proof}

 for $n\geq 6$, $P(S_n)$ contain a path $(p q r)(x y) \sim (x y) \sim ( q r z)(x y) \sim (q z r)$ and thus $P(S_6)$ is not a cograph. 

For $n\leq5$, the maximal cyclic subgroups intersect in the identity, and their
orders are in the sets $\{2\}$ (for $n=2$), $\{2,3\}$ (for $n=3$), $\{2,3,4\}$
(for $n=4$), or $\{4,5,6\}$ (for $n=5$), so these symmetric groups are all
power-cograph groups, by Theorem~\ref{t:maxcyclic}.
\end{proof}

\begin{theorem}
Let $p$ and $q\ (<p)$ be primes and $G$ be the semidirect product of  $C_{p}$ by $ C_{q^{m}}$ acting faithfully on $C_{p}$. Then $P(G)$ is a cograph.
\end{theorem}

\begin{proof}
By assumption, there are no elements of order $pq$, so the orders of the maximal
cyclic subgroups are $p$ and $q^m$.
\end{proof}

\begin{theorem}
If $G$ is a dihedral group of order $2m$, then $G$ is a power-cograph group if
and only if $m$ is either a prime power or the product of two distinct primes.
\end{theorem}

\begin{proof}
The orders of maximal cyclic subgroups are $2$ and $m$, and intersection of any two cyclic subgroup is the identity.
\end{proof}

\subsection{Remarks on solvable groups}

Let $G$ be a solvable group and $G\in \mathcal{C}$. Let $F(G)$ be the Fitting subgroup of $G$.  Then by Theorem \ref{t:nilp_cograph}, $F(G)$ is either of prime power order or a cyclic group of order $pq$, where $p$ and $q$ are distinct primes.

First, let $F(G)=C_{pq}$ for distinct primes $p$ and $q$. Then $F(G)$ contains its centraliser, and so is equal to it; so $G/F(G)$ acts as a group of automorphisms of $F(G)$. Thus $G$ is contained in the group $(C_p:C_{p-1})\times(C_q:C_{q-1})$. If $G$ contains a direct product larger than $C_p\times C_q$, then this
product is described by Theorem~\ref{product:cographs}: it has the form
$(C_p:C_{q^m})\times C_q$. If $G$ is strictly larger than this, then it
contains an element of prime order $r$ with $r\mid p-1$ and $r\mid q-1$, acting
non-trivially on both $C_p$ and $C_q$. But then $G$ contains a subgroup
$C_q\times(C_q:C_r)$, contrary to Theorem~\ref{product:cographs}.

Otherwise the structure of $G$ is $(C_p\times C_q).C_r$ where $r$ divides both $p-1$ and $q-1$, and $r$ is either a prime power or the product of two primes.
Such a group is a PCG group, since its maximal cyclic subgroups have orders
$pq$ or $r$.

\medskip

Next suppose that $F(G)$ be a $p$-group. We divide this case into two subcases.
 
 If all the elements of $G$ are of prime power order then the prime graph of $G$ is a null graph, and hence $G\in \mathcal{C}$. Higman \cite{Higman} gave a nice characterization of such groups. And in that case $|G|$ has at most two prime divisors and $G/F(G)$ is one of the following:
 \begin{enumerate}
     \item 
 a cyclic group whose order is a power of a prime other than $p$.
 \item 
 a generalized quaternion group, $p$ being odd; or 
 \item a group of order $p^aq^b$
with cyclic Sylow subgroups, $q$ being a prime of the form $kp^a+1$.
 \end{enumerate}

But difficulties arise when $F(G)$ is a $p$-group and $G$  contains elements whose order is not a prime power. By Theorem \ref{t:nilp_cograph}, the order of any element in a group in $\mathcal{C}$ is either
a prime power or the product of two primes. This case can occur; here
are two examples: 
\begin{example}
The Frobenius group $F_7$ of order 42 has $P(F_7)$ a cograph. Here $|F_7|$ is divisible by 3 primes, $F_7$ contains an element of order 6, and it's Fitting subgroup $C_7$.
\end{example}
\begin{example}
Let $G$ be the semidirect product of the Heisenberg group $H_3$ of order 27 by $C_2$. Then $G$ is solvable and $G\in \mathcal{C}$. In this case, the Fitting subgroup $F(G)=H_3$, and $G$ contains elements of order 6.
\end{example} 
 
 \begin{problem}
Classify all solvable $\mathcal{C}$-groups whose Fitting subgroup is a $p$-group.
 \end{problem}

 \section{Finite simple groups}

 In this section we discuss simple groups whose power graphs are cographs. For each prime $p$, the simple group $C_p$ has complete power graph, therefore it is a power-cograph group. In the next theorem we classify alternating groups which are power-cograph groups.

 \begin{theorem}
 \label{alter:cograph}
The alternating group $A_n$ is a power-cograph group if and only if $n\leq 6$.
\end{theorem}
 
\begin{proof}
For $n\ge7$, the $4$-$6$ test is applicable, with $a=(1,2,3,4)$ and
$b=(1,3)(2,4)(5,6,7)$.

Now we consider $n \leq 6$.

If $n=3$ then $A_3$ is nothing but the cyclic group $C_3$ and hence its power graph is the complete graph $K_3$ and hence a cograph.

For $n=4,5,6$ then prime graph of $A_n$ is a null graph and by Theorem \ref{prime:cograph} the power graph is a cograph.
\end{proof}

In the next few sections we discuss simple groups of Lie type of low rank or over small fields and sporadic simple groups.
Information about specific groups is found in the
$\mathbb{ATLAS}$~\cite{Conway}, and further information about the simple groups
and their subgroups is in Rob Wilson's book~\cite{wilson}.

We also use the fact that $\mathcal{C}$ is subgroup-closed; so, if a group $G$
contains a subgroup not in $\mathcal{C}$, then $G\notin\mathcal{C}$.

\subsection{Simple groups of Lie type of rank $1$}
The simple groups of Lie  type of rank $1$ are $A_1(q)=\mathrm{PSL}(2,q)$,
$^2A_2(q)=\PSU(3,q)$, $^2B_2(q)=\Sz(q)$ where $q=2^{2e+1}$, and
$^2G_2(q)=R_1(q)$ where $q=3^{2e+1}$.

In \cite{Cameron}, Cameron proved that, if $q$ is an odd prime power, then the power graph of $\PSL(2,q)$ is a cograph if a only if $(q-1)/2$ and $(q+1)/2$ are either prime powers or product of two primes. And if $q\geq 4$ is a power of 2 then the power graph of $\PSL(2,q)$ is a cograph if and only if $q-1$ and $q+1$ are either prime
powers or products of two distinct primes.

Next we show that power graph of $\PSU(3,q)$ is not a cograph for $q\neq 2$.
Since $\PSU(3,2)$ is not simple, there are no simple power-cograph groups
of this type.

\begin{theorem}
Let $q$ be a power of a odd prime $p$. Then power graph of $\PSU(3,q)$ is not a cograph.
\end{theorem}

\begin{proof}
We use the fact that $\PSU(3,q)$, $q$ odd, has a cyclic subgroups of order
$(q^2-1)/\gcd(q+1,3)=(q-1)\cdot(q+1)/\gcd(q+1,3)$.
So, if the power graph is a cograph, then both $(q-1)$ and
$(q+1)/\gcd(q+1,3)$ are primes, or else both are powers of the same prime.
But both these numbers are even; so they must both be powers of $2$.
Since one of $q-1$ and $q+1$ is not divisible by $4$, we must have
$(q-1,q+1)=(2,4)$ or $(4,6)$, so $q=3$ or $5$.

Now for $q=3$ the group $\PSU(3,3)$ contains elements of order~$12$, so the power graph is not a cograph. On the other hand $\PSU(3,5)$ contains $A_7$. Therefore the power graph of $\PSU(3,q)$ is not a cograph.
\end{proof}

\begin{theorem}
	Let $q\geq 4$ be a power of  $2$. Then the power graph of $\PSU(3,q)$ is not a cograph.
\end{theorem}

\begin{proof} Let $\beta$ be a generator of the multiplicative group of
$\mathrm{GF}(q^2)$. Then $\beta^{q-1}$ has order $q+1$. Let $p$ be a prime factor of $q+1$ greater than~$3$, and let $d=(q+1)/p$. Then $\alpha=\beta^{d(q-1)}$ has order $p$. Then $\overline{\alpha}=\beta^{d(q^2-q)}$, so $\alpha\overline{\alpha}=\beta^{d(q^2-1)}=1$ in $\mathrm{GF}(q^2)$. Consider the elements
$$g=\begin{pmatrix}
0&1&0\\
1&0&0\\
0&0&1\\
\end{pmatrix}$$
$$h=\begin{pmatrix}
\alpha&0&0\\
0&\alpha&0\\
0&0&\alpha^{-2}\\
\end{pmatrix}$$
$$k=\begin{pmatrix}
0&\alpha&0\\
1&0&0\\
0&0&\alpha^{-1}\\
\end{pmatrix}$$
Then $g$ is a element of order 2 and it commutes with $h$. So $o(gh)=o(hk)=2p$. On the other hand $k^2=h$.

Therefore the elements  the elements $g,gh,h,k$ induce a path of length~$3$ in 
$\SU(2,q)$.

Now observe that $g,h,k\in\SU(3,q)\setminus Z$. Take $\gamma=gZ$ and $\eta=hZ$, $\kappa=kZ$. Then the elements
$\gamma,\gamma\eta,\eta,\kappa$ induce a path of length~$3$ in the power graph of $\PSU(3,q)$.

The argument fails for $q=8$.  But we saw earlier that $\PSU(3,8)$ is not a
power-cograph group, using the $4$-$6$ test.
\end{proof}

\begin{theorem}
Let $G={}^2B_2(q)=\Sz(q)$,  $q=2^{2e+1}$. Then $G\in\mathcal{C}$ if and only if each of
$q-1$, $q+\sqrt{2q}+1$ and $q-\sqrt{2q}+1$ is either a prime power or the
product of two distinct primes.
\end{theorem}

\begin{proof}
Any edge of the power graph is contained in a maximal cyclic subgroup.
The maximal cyclic subgroups of $\Sz(q)$ have orders
$4$, $q-1$, $q+\sqrt{2q}+1$ and $q-\sqrt{2q}+1$. These four numbers are
pairwise coprime. (The last three are odd. The difference between the third
and fourth is a power of $2$, but $2$ does not divide either. Suppose that
$p$ is a prime dividing both $q-1=2^{2e+1}-1$ and
$q+\sqrt{2q}+1=2^{2e+1}+2^{e+1}+1$. Then $p$ divides their difference,
$2^{e+1}+2$; since it is odd, it divides $2^e+1$, and hence it divides
$2^{2e}-1$, and also $2^{2e+1}-2$. This $p$ divides $1$. The argument for
$q-1$ and $q-\sqrt{2q}+1$ is similar.) Thus no element can lie in maximal
cyclic subgroups of different orders. So, if the power graph contains 
$P_4$, then this $P_4$ must be contained in a maximal cyclic subgroup, so
this subgroup must have three prime divisors, not all equal. The converse
is clear.
\end{proof}

Now let $G={}^2G_2(q)=R_1(q)$, $q=3^{2e+1}$.
The centraliser of an involution in $G$ is $C_2\times\PSL(2,q)$, which contains
subgroups $C_2\times C_{(q\pm1)/2}$. So, if $G\in\mathcal{C}$, then
$(q\pm1)/2$ is either prime or a power of $2$. If it is a power of $2$, then
we have a solution to Catalan's equation, contradicting the result of
Mih\u{a}ilescu’s Theorem: see \cite[Section 6.11]{Cohn}.
The numbers $(q\pm1)/2$ have opposite parity, so cannot both be prime. So
$G$ is not a power-cograph group.

\subsection{Simple groups of Lie type of rank 2}

The rank 2 simple groups of Lie type are
$A_2(q)=\PSL(3,q)$, $C_2(q)=\PSp(4,q)$, ${}^2A_3(q)=\PSU(4,q)$, ${}^2A_4(q)=\PSU(5,q)$, $G_2(q)$, ${}^2F_4(q)$ and ${}^3D_4(q)$. We examine each of the above cases. In the case of $A_2(q)$, we prove a slightly stronger result, for later use.

\begin{theorem}
Let $G$ be a quotient of $\SL(3,q)$ by a subgroup of the group of scalars.
If $G$ is a power-cograph group, then $q=2$ or $q=4$.
\end{theorem}

\begin{proof}
We work in $\mathrm{SL}(3,q)$. Suppose that $q$ is odd. Consider the elements
\[g=\begin{pmatrix}
0&-1&0\\
1&0&0\\
0&0&1\\
\end{pmatrix}
\qquad
h=\begin{pmatrix}
0&-1&0\\
1&1&0\\
0&0&1\\
\end{pmatrix}
\]
It is easily checked that
\[g^2 =h^3=\begin{pmatrix}
-I_2&O\\
 O&1 \\
\end{pmatrix}\]

So
$(h^2,h,h^3=g^2,g)$ is an induced path of length $3$ in the power graph.

Now observe that neither $g$ nor $h$ contains any non-identity scalar matrix.
So these elements project onto elements with the same property in the quotient
when a group of scalars is factored out.

\medskip

Now we consider $q$ to be a power of 2, with $q>4$. If $q$ is an odd power
of $2$, then $q-1$ is not divisible by $3$, while if $q$ is an even power
of $2$, then $q-1$ cannot be a power of $3$ (according to the solution of
Catalan's equation) and so must have a larger prime divisor.

Let $\alpha$ be an element of the multiplicative group of
$\mathrm{GF}(q)$ of prime order $p$ greater than $3$. Consider the elements
\[g=\begin{pmatrix}
1&1&0\\
0&1&1\\
0&0&1\\
\end{pmatrix}
\qquad
k=\begin{pmatrix}
\alpha&0&0\\
0&\alpha^{-2}&0\\
0&0&\alpha\\
\end{pmatrix}.\]
It is routine to check that
\[g^2=\begin{pmatrix}
1&0&1\\
0&1&0\\
0&0&1\\
\end{pmatrix},\]
an element of order~$2$; and that $g^2$ commutes with $k$, so that $g^2k$ has
order $2p$, and $(g^2k)^p=g^2$.

Putting $h=g^2k$, we have $h^p=g^2$, so the elements
$g,g^2=h^p,h,h^2$ induce a path of length~$3$.

No power of any of these elements except the identity is a scalar. (For this
we need $p>3$, since if $p=3$ then $\alpha^{-2}=\alpha$.) So factoring out
a group of scalars we get elements with the same properties.

Finally we note that $\PSL(3,2)$ and $\PSL(3,4)$ are power-cographs (as their
Gruenberg--Kegel graphs are null). However, $\PSL(3,2)\cong\PSL(2,7)$, so
this group does not need to be included in the statement of the theorem.
\end{proof}

\begin{theorem}
Let $G=\PSp(4,q)$. Then $P(G)$ is not a cograph.
\end{theorem}
\begin{proof}
A $4$-dimensional symplectic space is the direct sum of two $2$-dimensional
symplectic spaces; and the $2$-dimensional symplectic group is the special
linear group. So $G=\PSp(4,q)$ contains a subgroup which is the direct product
of two copies of $\PSL(2,q)$ if $q$ is even, or the central product of two
copies of $\SL(2,q)$ if $q$ is odd.

Thus $G$ contains the direct product of cyclic groups of orders $q\pm1$ if
$q$ is even, and a quotient of this by a subgroup of order~$2$ if $q$ is odd.

For $q$ even, $q-1$ and $q+1$ are coprime, so $P(G)$ is a cograph only if 
both are primes; since one is divisible by $3$, this requires $q=2$ or
$q=4$.

For $q$ odd, one of $(q-1)/2$ and $(q+1)/2$ is even, so the order of the
cyclic subgroup is divisible by $4$ and (if $q>3$) by at least one further
prime. So $P(G)$ is a cograph only if $q=3$.

Now $\PSp(4,2)\cong S_6$ is not simple; $\PSp(4,3)$ contains elements of 
order~$12$; and $\PSp(4,4)$ is ruled out by the $4$-$6$ test.
\end{proof}

\begin{theorem}
The power graph of $G_2(q)$ is not a cograph.
\end{theorem}

\begin{proof}
The group $G_2(q)$ contains both $\SL(3,q)$ and $\SU(3,q)$
\cite{cooperstein, kleidman2}. Now
$\SL(3,q)=\PSL(3,q)$ if $q\not\equiv1\pmod{3}$, while $\SU(3,q)=\PSL(3,q)$ if
$q\not\equiv-1\pmod{3}$. So, for any $q$, $G_2(q)$ contains either $\PSL(3,q)$
or $\PSU(3,q)$. Now the former is in $\mathcal{C}$ only for $q=2$ or $q=4$,
and the latter is never in $\mathcal{C}$ except for $q=2$ (this group is not
simple). So the only case needing further consideration is $q=2$; but $G_2(2)$
is not simple, and is not in $\mathcal{C}$ (it contains $\PSU(3,3)$ as a
subgroup of index~$2$).
\end{proof} 
Below we give arguments for the rest of the simple groups of Lie type of
rank~$2$.
We find that in each of the following cases the power graph is not a cograph.
\begin{itemize}\itemsep0pt
\item
Let $G={}^2A_3(q)=\PSU(4,q)$. This group contains $\PSp(4,q)$, so we only need consider $q=2$. But
$\PSU(4,2)\cong\PSp(4,3)$.
\item 
The group $G={}^2A_4(q)=\PSU(5,q)$ contains $\PSU(4,q)$. So $G\notin \mathcal{C}$.
\item
The group ${}^2F_4(2^d)$ contains ${}^2F_4(2)$ for all odd $d$ 
(Malle~\cite{malle}), and ${}^2F_4(2)$ is ruled out by the $4$-$6$ test.
\item 
The group {$G={}^3D_4(q)$} contains $G_2(q)$ (see Kleidman~\cite{kleidman}).
\end{itemize}

\subsection{Higher rank}

Let $G$ be a simple group of Lie type of higher rank. We show that $P(G)$ is not a cograph. 

Since the Dynkin diagram of $G$ contains a single bond in all cases, $G$ has
a subgroup of a Levi factor which is a quotient of $\SL(3,q)$ by a group of
scalars. The results of the preceding section give the desired conclusion
if $q\notin\{2,4\}$.

It remains to deal with groups over the fields of $2$ or $4$ elements.

Now $\PSL(4,2)\cong A_8$, so its power graph is not a cograph,
while $\PSp(6,2)$ is excluded by the $4$-$6$ test. Moreover, $\PSL(4,4)$
contains $\PSL(4,2)$, and $\PSp(6,4)$ contains $\PSp(6,2)$ (by restricting
scalars). The orthogonal and unitary groups of Lie rank~$3$ all contain
$\PSp(4,q)$ for $q=2$ or $q=4$. So $P(G)$ is not a cograph.

\subsection{Sporadic simple groups}

Now we prove that there exist no sporadic simple group whose power graph is cograph. Recall that there are 26 sporadic simple groups \cite{Conway}, namely,  the five Mathieu groups ($M_{11}$, $M_{12}$, $M_{22}$, $M_{23}$ and $M_{24}$), four Janko groups ($J_1$, $J_2$, $J_3$ and $J_4$), three Conway groups ($Co_1$, $Co_2$ and $Co_3$), three Fischer groups ($Fi_{22}$, $Fi_{23}$ and $Fi_{24}$), 
Higman–Sims group ($HS$), the McLaughlin group ($M^cL$), the Held group $He$,
the Rudvalis group ($Ru$), the Suzuki group ($Suz$), the O'Nan group ($O'N$), the Harada–Norton group $HN$, the Lyons group ($Ly$), the Thompson group ($Th$)
the Baby Monster group ($B$) and the Monster group ($M$). Amongst these 26 groups the the Mathieu group $M_{11}$ is of smallest order ($|M_{11}|=7920=2^4 \cdot 3^2\cdot 5 \cdot 11 $). 

\begin{observation}
\label{Mathiew:cograph}
We observe, using information in the $\mathbb{ATLAS}$ of Finite Groups
\cite{Conway}, that $M_{11}$ is not a power-cograph group, by the $4$-$6$ test;
it contains elements $a,b$ of orders $4$ and $6$ respectively with $a^2=b^3$.
\end{observation}

   \begin{theorem}
   Let $G$ be a sporadic simple group. Then $P(G)$ is not a cograph.
   \end{theorem}

 \begin{proof}
 Observation \ref{Mathiew:cograph} shows that the power graph of the Mathieu group $M_{11}$ is not a cograph. Now the Mathieu group $M_{11}$ is a subgroup of
all the other sporadic simple groups except $J_1$, $M_{22}$, $J_2$, $J_3$, $He$, $Ru$ and $Th$. So the power graphs of these groups are also not cographs.

For the other seven groups we look for subgroups which are not 
power-cograph groups. We observe that $J_1$ contains $D_3\times D_5$, $M_{22}$ contains $A_7$, $J_2$ contains $A_4\times A_5$, $J_3$ contains $C_3\times A_6$, $He$ contains $S_7$, $Ru$ contains $A_8$ and $Th$ contains $\PSL(2,19):C_2$.
By Theorems \ref{product:cographs}, \ref{symm:cograph} and \ref{alter:cograph},
the power graphs of these subgroups are not cographs. Hence the power graphs of the original groups are not cographs. 
 \end{proof}

   \subsection*{Acknowledgements}
The authors are thankful to discussion sessions in ``Research Discussion on Graphs and Groups" organized by Cochin University of Science and Technology, India.
The author Pallabi Manna is supported by CSIR (Grant No-09/983(0037)/2019-EMR-I). Ranjit Mehatari thanks the SERB, India, for financial support (File Number: CRG/2020/000447) through the Core Research Grant.

\end{document}